\documentclass[11pt]{amsproc}

\usepackage[top=1in,bottom=1in,left=1.5in,right=1.5in]{geometry}
\usepackage{amsfonts} %\mathbb{R}, etc.
\usepackage{amsmath}
\usepackage{amssymb}
\usepackage{amsthm}
\usepackage{setspace}
\usepackage{subfigure}
\usepackage{url}
\usepackage{booktabs}
\usepackage{ifthen}
\usepackage{tikz}
\usetikzlibrary{calc}
\usepackage{enumerate}

\newcommand{\A}{\mathcal{A}}

\newcommand{\D}{\mathcal{D}}

\newtheorem{thm}{Theorem}[section]
\newtheorem{lem}[thm]{Lemma}
\newtheorem{prop}[thm]{Proposition}
\newtheorem{cor}[thm]{Corollary}

\newtheorem{defn}[thm]{Definition}

\numberwithin{equation}{section}

\begin{document}

\title{A class of graphs approaching Vizing's conjecture}

\author{Aziz Contractor}
\address{Aziz Contractor (\tt acontractor@student.clayton.edu)}

\author{Elliot Krop}
\address{Elliot Krop (\tt elliotkrop@clayton.edu)}
\address{Department of Mathematics, Clayton State University}
\date{\today}

\maketitle

\begin {abstract}
For any graph $G=(V,E)$, a subset $S\subseteq V$ \emph{dominates} $G$ if all vertices are contained in the closed neighborhood of $S$, that is $N[S]=V$. The minimum cardinality over all such $S$ is called the domination number, written $\gamma(G)$. In 1963, V.G. Vizing conjectured that $\gamma(G \square H) \geq \gamma(G)\gamma(H)$ where $\square$ stands for the Cartesian product of graphs. In this note, we define classes of graphs $\A_n$, for $n\geq 0$, so that every graph belongs to some such class, and $\A_0$ corresponds to class $A$ of Bartsalkin and German. We prove that for any graph $G$ in class $\A_1$, $\gamma(G\square H)\geq \left(\gamma(G)-\sqrt{\gamma(G)}\right)\gamma(H)$.
\\[\baselineskip] 2010 Mathematics Subject
      Classification: 05C69
\\[\baselineskip]
\end {abstract}

 \section{Introduction}
 For basic graph theoretic notation and definitions see Diestel~\cite{Diest}. All graphs $G(V,E)$ are finite, simple, connected, undirected graphs with vertex set $V$ and edge set $E$. We may refer to the vertex set and edge set of $G$ as $V(G)$ and $E(G)$, respectively.
 
 For any graph $G=(V,E)$, a subset $S\subseteq V$ \emph{dominates} $G$ if $N[S]=V(G)$. The minimum cardinality of $S \subseteq V$, so that $S$ dominates $G$ is called the \emph{domination number} of $G$ and is denoted $\gamma(G)$. We call a dominating set that realizes the domination number a $\gamma$-set.

 \begin{defn}
 The \emph{Cartesian product} of two graphs $G_1(V_1,E_1)$ and $G_2(V_2,E_2)$, denoted by $G_1 \square G_2$, is a graph with vertex set $V_1 \times V_2$ and edge set $E(G_1 \square G_2) = \{((u_1,v_1),(u_2,v_2)) : v_1=v_2 \mbox{ and } (u_1,u_2) \in E_1, \mbox{ or } u_1 = u_2 \mbox{ and } (v_1,v_2) \in E_2\}$.
\end{defn}

For a vertex $h\in V(H)$, the $G$-fiber, $G^h$, is the subgraph of $G\square H$ induced by $\{(g,h):g\in V(G)\}$. Similarly, for a vertex $g\in V(G)$, the $H$-fiber, $H^g$, is the subgraph of $G\square H$ induced by $\{(g,h):h\in V(H)\}$.

Perhaps the most popular and elusive conjecture about the domination of graphs is due to Vadim G. Vizing (1963) \cite{Vizing}, which states
\begin{align}
\gamma(G \square H) \geq \gamma(G)\gamma(H).\label{V}
\end{align}

To read more about past attacks on the conjecture, and which graphs are known to satisfy its statement, see the survey \cite{BDGHHKR}.

One of the earliest significant results is that of Bartsalkin and German \cite{BG}, who showed that the conjecture holds for decomposable graphs, that is, graphs $G$ with vertex sets which can be disjointly covered by $\gamma(G)$ cliques, as well as all spanning subgraphs of decomposable graphs with the same domination number. Bartsalking and German called the family of such graphs class $A$. There are known examples of graphs not in class $A$, see for example \cite{BDGHHKR} page $5$, however, the examples in the literature satisfy the property that if we add the maximum number of edges to such graphs without changing the domination number, the clique number is one more than the domination number of the resulting graph. This gives motivation to consider such graphs for Vizing's conjecture.

Furthermore, it is interesting to generalize to the class of decomposable graphs to those with clique number exceeding the domination number by some fixed amount, since every graph falls into some such class. By producing bounds on the domination numbers of cartesian products of graphs where one is in such a class, we could hope to produce a better bound for all graphs. 

The best current bound for the conjectured inequality was shown in 2010 by Suen and Tarr \cite{ST}, 
\[\gamma(G \square H) \geq \frac{1}{2}\gamma(G)\gamma(H)+\frac{1}{2}\min\{\gamma(G),\gamma(H)\}.\]

In this note, we extend the technique of Bartsalkin and German, defining classes of graphs $\A_n$ for $n\geq 0$, and show that for any $G$ in class $\A_1$, \[\gamma(G\square H)\geq \left(\gamma(G)-\sqrt{\gamma(G)}\right)\gamma(H)\]

Although graphs in classes $\A_n$ for $n>1$ are not well understood, Douglas Rall has produced examples for $\A_{2n-4}$ for any $n\geq 2$ (personal communications).

%Let n be a positive integer that is at least 3. For each integer k such that 1 <= k <= n, let V_k = { a_k, b_k } and let W_k = { c_k, d_k }. Let V be the union of all V_k and let W be the union of all W_k.  The vertex set of the graph G_n is  V \cup W.  The graph G_n is bipartite (with parts V and W) and has the following edges.  For each k, the open neighborhood of a_k is every vertex in W except for c_k and  d_k. The vertex b_k  has the same open neighborhood as a_k.  (An easier way to think of this graph G_n is that it is the complete bipartite graph  K_{2n,2n}  with the edges of n vertex disjoint 4-cycles removed.)

%I think it is the case that  \gamma(G_n)=4.  In addition, you can show that if any missing edge is added to G_n, then the domination number of the resulting graph is strictly less than 4.   In other words, if F is any graph for which G_n is a spanning subgraph and F is not equal to  G_n, then \gamma(F)  <  \gamma(G_n).  That is, G_n is edge-critical with respect to domination.  A consequence of this is that cliques in G_n are single edges.  It now follows that \theta(G_n) = 2n = 4 + (2n-4) = gamma(G_n) + (2n-4).  Thus, G_n belongs to the class  A_{2n-4}.

We adhere closely to the notation of \cite{BDGHHKR}.

\section{Extending the Argument of Bartsalkin-German}

\subsection{Concepts and Consequences}

Given a graph $G$, we say that $G$ \emph{satisfies Vizing's conjecture} if for any graph $H$, \eqref{V} holds.
%
%The following fact was shown in \cite{BG}.
%
%\begin{prop}
%Let $G$ be a graph that satisfies Vizing's conjecture, and let $G'$ be a spanning subgraph of $G$ such that $\gamma(G')=\gamma(G)$. Then $G'$ also satisfies Vizing's conjecture.
%\end{prop}

The clique covering number $\theta(G)$ is the minimum number $k$ of sets in a partition $\mathcal{C}=\{V_1 \cup \dots \cup V_k\}$ of $V(G)$ such that each induced subgraph $G[V_i]$ is complete.

The following is a recursive definition of the class of graphs $\D_n$. Let $\D_0$ be the set of decomposable graphs of Bartsalkin and German \cite{BG}, that is, those graphs $G$ so that $\theta(G)=\gamma(G)$.

\begin{defn}\label{Dn}
For any positive integer $n$ let $\D_n$ be the class of graphs $G$ such that $\theta(G)= \gamma(G)+n$ and $G$ is not the spanning subgraph of any graph $H\in \D_m$ for $0\leq m <n$ such that $\gamma(G)=\gamma(H)$.
\end{defn}

\begin{defn}\label{An}
For any non-negative integer $n$, let $\A_n$ be the class of graphs $G$, such that $G$ is a spanning subgraph of some graph $H\in \D_n$ so that $\gamma(G)=\gamma(H)$.
\end{defn}

Thus, Bartsalkin and German showed that graphs in class $\A_0$ satisfy Vizing's conjecture.

A known example \cite{BDGHHKR} of a graph not in $\A_0$ is $K_{6,6}$ with the edges of $3$ vertex-disjoint $4$-cycles removed. It is not difficult to check that it is in $\A_1$.

The following is another example \cite{BDGHHKR} of a graph not in $\A_0$ and in $\A_1$.

\pagebreak

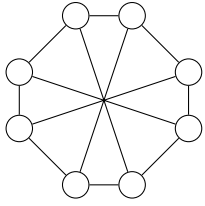
\begin{figure}[ht]
\begin{center}
\begin{tikzpicture}[scale=.75]
 
	 \tikzstyle{vertex}=[circle,draw, minimum size=10pt,inner sep=0pt]
	 \tikzstyle{selected vertex} = [vertex, fill=red!24]
	 \tikzstyle{selected edge} = [draw,line width=5pt,-,red!50]
	 \tikzstyle{edge} = [draw,thick,-,black]

	\node[vertex](u1) at (1,1)[]{};
	\node[vertex](u2) at (2,1)[]{};
	\node[vertex](u3) at (3,2)[]{};
	\node[vertex](u4) at (3,3)[]{};
	\node[vertex](u5) at (2,4)[]{};
	\node[vertex](u6) at (1,4)[]{};
	\node[vertex](u7) at (0,3)[]{};
	\node[vertex](u8) at (0,2)[]{};

          \draw[color=black] 
          (u1)--(u2)--(u3)--(u4)--(u5)--(u6)--(u7)--(u8)--(u1) (u1)--(u5) (u2)--(u6) (u3)--(u7) (u4)--(u8);

\end{tikzpicture}
\end{center}
\caption{a known graph not in $\A_0$ }
\end{figure}

The next two observations are generalized from \cite{BG}.

\begin{lem}\label{tool}
For a chosen non-negative integer $n$, let $G$ be a graph in class $\D_n$ with $\gamma(G)=k$ and $\mathcal{C}=\{C_1,\dots, C_{k+n}\}$ the clique partition of $V(G)$. For any non-negative integer $l<k+n$ and $C_{i_1},\dots, C_{i_l}\in \mathcal{C}$, let $D$ be a smallest set of vertices from $V(G)-(C_{i_1}\cup\dots\cup C_{i_l})$ that dominates $C_{i_1}\cup\dots\cup C_{i_l}$. Let $C_{j_1},\dots, C_{j_t}$ be the cliques from $\mathcal{C}$ that have a non-empty intersection with $D$. Then
\begin{align}
\sum_{m=1}^t(|C_{j_m}\cap D|-1)\geq l-n \label{CK}
\end{align}
\end{lem}

\begin{proof}
The proof is a simple application of the pigeonhole principle.
\end{proof}

\begin{prop}\label{spanning}
Let $G$ be a graph so that for any $H$, $\gamma(G\square H)\geq f(\gamma(G),\gamma(H))$ for some function $f$ of $\gamma(G), \gamma(H)$. For any spanning subgraph $G'$of $G$ such that $\gamma(G')=\gamma(G)$, $\gamma(G'\square H)\geq f(\gamma(G'),\gamma(H))$.
\end{prop}

\begin{proof}
Notice that since $G'\square H$ is an spanning subgraph of $G\square H$, $\gamma(G'\square H)\geq \gamma(G\square H)$. Furthermore, $f(\gamma(G),\gamma(H))=f(\gamma(G'),\gamma(H))$.
\end{proof}

We now introduce the main concept which allows us to work within the classes.

\begin{defn}
For any clique partition $\mathcal{C}=\{C_1, \dots, C_m\}$ of $V(G)$, if some vertex $v\in V(G)$, dominates $j$ cliques of the partition, for some $j$, $1\leq j \leq m$, then we say that $v$ is \emph{$j$-restraining}. 
\end{defn}

\begin{defn}\label{rest}
For any integers $l,m,$ clique partition $\mathcal{C}=\{C_1, \dots, C_m\}$ of $V(G)$, and $C_{i_1},\dots, C_{i_l}\in \mathcal{C}$, suppose a set of vertices $D\subseteq V(G)-(C_{i_1}\cup\dots\cup C_{i_l})$ dominates $C_{i_1},\dots, C_{i_l}$. If $C_{j_1}, \dots, C_{j_t}$ are the cliques from $\mathcal{C}$ that have a non-empty intersection with $D$, then we say $D$ is \emph{$(\left|D\right|,l+t)$-restraining}. We say that $l+t$ is the \emph{restraint} of $D$ and $l+t-\left|D\right|$ the \emph{excess} of $D$.
\end{defn}

Notice that a vertex which is $j$-restraining is also $(1,j)$-restraining.

\medskip

The next lemma describes how the sum of restraint of a graph in class $\D_n$ is limited by $n$.

\begin{lem}\label{restraining}
For any non-negative integer $n$, let $G$ be a graph in class $\D_n$ with $\gamma(G)=k$ and $\mathcal{C}=\{C_1,\dots, C_{k+n}\}$ the clique partition of $V(G)$. Suppose for some non-negative integers $l$ and $t$, that $D$ is a $(\left|D\right|, l+t)$-restraining set, dominating $\mathcal{C_D}=\{C_{i_1},\dots, C_{i_l}, C_{j_1}, \dots, C_{j_t}\}$ as in Definition \ref{rest}. Then $G- \mathcal{C_D}$ cannot contain a set $E$ of vertices which is $(\left|E\right|, \left|D\right|+\left|E\right|+n-(l+t)+1)$-restraining.
\end{lem}
\begin{proof}
If we suppose to the contrary, then $D \cup E$ is a set of vertices dominating $l+t+\left|D\right|+\left|E\right|+n-(l+t)+1=\left|D\right|+\left|E\right|+n+1$ cliques of the partition. We count $\left|D \cup E\right|$ and one vertex from each undominated clique and find a dominating set of $G$ with size at most $k-1$, which is a contradiction.
\end{proof}

\begin{defn}\label{fiber}
For a vertex $h\in V(H)$, the \emph{$G$-fiber}, $G^h$, is the subgraph of $G\square H$ induced by $\{(g,h):g\in V(G)\}$. Similarly, for a vertex $g\in V(G)$, the $H$-fiber, $H^g$, is the subgraph of $G\square H$ induced by $\{(g,h):h\in V(H)\}$.
\end{defn}

\begin{defn}\label{cell}
For any minimum dominating set $D$ of $G\square H$ let $C_1,\dots, C_{k+n}$ be a clique partition of $V(G)$. For every $h\in V(H)$ and $i,\, 1\leq i \leq k+n$, we call $C_i^h=C_i\times \{h\}$ a \emph{G-cell}.
\end{defn}

\begin{defn}\label{missing}
If $D\cap(C_i\times N[h])$ is empty, call such $C_i^h$ a \emph{missing G-cell} for $h$. 
\end{defn}

Notice that every missing $G$-cell for $h$ is dominated ``horizontally", in $G^h$. We often write $C_{i_1}^h,\dots, C_{i_l}^h$ as the missing $G$-cells for $h$ with vertices dominated from $C_{j_1}^h,\dots, C_{j_t}^h$. 

\medskip

For a clique partition of $G$ and minimum dominating set $D$ of $G\square H$, we define a labeling of vertices in $D$, similar to that of \cite{BG} which we call the {\bf\emph{simple labeling}}:

For $G\in \D_n, \gamma(G)=k, 1\leq i \leq k+n$, and $h \in V(H)$, if $D \cap C_i^h$ is non-empty, we label all of those vertices by $i$. Choose any vertex $h \in V(H)$. If there exist vertices in $D\cap(C_i\times N[h])$, then one of them received the label $i$. Notice that projecting all vertices labeled $i$ onto $H$ produces a vertex-labeling where $h$ is adjacent to a vertex labeled $i$.

\subsection{The Argument}
For our main result, our reasoning can be divided into two counting arguments which we call the \emph{Undercount Argument} and the \emph{Overcount Argument}. As in the method of Bartsalkin and German, we label vertices of the minimum dominating set $D$ of $G \square H$ by the label of the clique containing their projection onto $G$. In the undercount argument, we remove some of these labels from all vertices of $D$, which allows us to relabel them. In the overcount argument, for certain fibers $G^h$, we assign multiple labels to one vertex of $D$ in each fiber, and later remove the resulting overcount.

\begin{thm}\label{Viz}
For any graphs $G\in \A_1$ and any $H$,
\[\gamma(G\square H)\geq \left(\gamma(G)-\sqrt{\gamma(G)}\right)\gamma(H)\]
\end{thm}

\begin{proof}
Suppose $G\in \D_1$ with $\gamma(G)=k$. For any graph $H$, and a minimum dominating set $D$ of $G\square H$ let $C_1,\dots, C_{k+1}$ be a clique partition of $V(G)$. 
%For every $h\in V(H)$ and $i,\, 1\leq i \leq k+1$, we call $C_i^h=C_i\times \{h\}$ a \emph{G-cell}.
\\[\baselineskip]
\emph{Undercount Argument:}
\\[\baselineskip]
Suppose $G$ has minimum restraint $2$ and contains a $2$-restraining vertex $v$, and without loss of generality, suppose $v\in C_1$ with $C_1\cup C_2 \subseteq N[v]$. 

%We re-partition $V(G)$ by adding to $D_2$ all vertices not in $D_1 \cup D_2$ which dominate $D_2$ and then adding to $D_1$ all vertices not in $D_1 \cup D_2$ which dominate $D_1$. Call the resulting cliques $C_2$ and $C_1$, respectively, and for $3\leq i \leq k+1$, rename the clique formerly $D_i$ as $C_i$. Thus, $C_2$ and $C_1$ are maximal cliques with respect to $C_3, \dots, C_{k+1}$. Notice that if $C_i$ is empty for any $i$, $1\leq i \leq k+1$, then $G$ is decomposable and not in class $\A_1$, which is a contradiction.

By Lemma \ref{restraining}, $C_3\cup\dots\cup C_{k+1}$ cannot contain a set of vertices $E$ which is \\* $(\left|E\right|, \left|E\right|+1)$-restraining. Thus, the induced subgraph on $C_3\cup \dots \cup C_{k+1}$ satisfies formula \eqref{CK} with $n=0$.

We consider only missing cells in $\cup_{i=3}^{k+1}C_i^h$. If $C_{j_1}^h,\dots, C_{j_t}^h \subseteq \cup_{i=3}^{k+1}C_i$, then applying Lemma \ref{tool} with $n=0$ we see that there are at most $l$ vertices in $D \cap V(G^h)$ with duplicated labels held by vertices of $D$ in the same cells. We relabel these vertices by assigning a label of a distinct corresponding missing cell. That is, at most $l$ vertices with duplicated labels receive labels $i_1, \dots, i_l$. 

If $C_1^h$ or $C_2^h$ are members of $\{C_{j_1}^h,\dots, C_{j_t}^h\}$, then applying Lemma \ref{tool},
\[\sum_{m=1}^t(|C_{j_m}^h\cap D|-1)\geq l-1.\] 
Thus, for every vertex $h\in V(H)$ and missing $G$-cells for $h$, $C_{i_1}^h,\dots, C_{i_l}^h$, there are $l-1$ vertices in $D\cap V(G^h)$ which have duplicated labels held by other vertices of $D$ in the same cell. By assumption, some of the dominating vertices of $C_{i_1}^h,\dots, C_{i_l}^h$ are from $C_1^h$ or $C_2^h$, and we remove the labels on those vertices; that is we remove the labels $1$ or $2$ from vertices of $D$ in $C_1^h$ and $C_2^h$. This produces at least $l$ vertices in $D \cap V(G^h)$ with duplicated labels held by other vertices of $D$ in the same cells. Now, for every missing cell in $G^h$, there are enough such duplicated labeled vertices so that every such vertex can be relabeled and receive a label of a distinct missing cell. That is, all missing cells are covered. Projecting all vertices with a given label greater than $2$ onto $H$ produces a dominating set of $H$, of size at least $\gamma(H)$. Summing over all labels, we count $(\gamma(G)-1)\gamma(H)$ vertices of $D$.

We repeat the argument assuming that $G$ contains a restraining set $S$, $\left|S\right|=r$, for some $r\in \{2, \dots, \gamma(G)\}$ which is $(r,r+1)$-restraining where $r+1$ is the minimum restraint of $G$. Without loss of generality, assume that $N[S]\supseteq C_1\cup \dots \cup C_{r+1}$.

By Lemma \ref{restraining}, $C_{r+2}\cup\dots\cup C_{k+1}$ cannot contain a set of vertices $E$ which is \\* $(\left|E\right|, \left|E\right|+1)$-restraining. Thus, the induced subgraph on $C_{r+2}\cup \dots \cup C_{k+1}$ satisfies formula \eqref{CK} with $n=0$.

For $1\leq i \leq k+1$ and $h \in V(H)$, if $D \cap C_i^h$ is non-empty, we label one of those vertices by $i$.

Let $C_{i_1}^h,\dots, C_{i_l}^h$ be the missing $G$-cells for $h$ with vertices dominated from \linebreak $C_{j_1}^h,\dots, C_{j_t}^h$. 

We consider only missing cells in $\cup_{i=r+2}^{k+1}C_i^h$. If $C_{j_1}^h,\dots, C_{j_t}^h \subseteq \cup_{i=r+2}^{k+1}C_i$, then applying Lemma \ref{tool} with $n=0$ we see that there are at most $l$ vertices in $D \cap V(G^h)$ with duplicated labels held by vertices of $D$ in the same cell. We relabel these vertices by assigning a label of a distinct corresponding missing cell. That is, at most $l$ unlabeled vertices receive labels $i_1, \dots, i_l$ and all missing cells are covered.

If any of $C_1^h, C_2^h, \dots, C_{r+1}^h$ are members of $\{C_{j_1}^h,\dots, C_{j_t}^h\}$, then applying Lemma \ref{tool},
\[\sum_{m=1}^t(|C_{j_m}^h\cap D|-1)\geq l-1.\] 

 By assumption, some of the dominating vertices of $C_{i_1}^h,\dots, C_{i_l}^h$ are from $C_1^h \cup C_2^h, \dots, \cup C_{r+1}^h$, and we remove the labels on those vertices; that is we remove the labels $1,2,\dots, r+1$ from vertices of $D$ in $C_1^h, \dots, C_{r+1}^h$. This produces at least $l$ vertices in $D \cap V(G^h)$ with duplicated labels held by vertices of $D$ in the same cells. Now, for every missing cell in $G^h$, there are enough such vertices with duplicated labels so that every such vertex can receive a label of a distinct missing cell. Projecting all vertices with a given label greater than $r+1$ onto $H$ produces a dominating set of $H$, of size at least $\gamma(H)$. Summing over all labels, we count 
\begin{align}
(\gamma(G)-r)\gamma(H)\label{form1}
\end{align}
 vertices of $D$.
%\vspace{.5 cm}
\\[\baselineskip]
\emph{Overcount Argument:}
\\[\baselineskip]
Next we condition on the minimum restraint of a vertex set in $G$ with excess $1$. Suppose $G$ has minimum restraint $r+1$ for some $1\leq r \leq \gamma(G)$, and let $E$ be a $(r,r+1)$-restraining set of vertices with minimum restraint $r+1$. For any $h\in H$, if $G^h$ contains a missing cell, then  $\left|D\cap V(G^h)\right|\geq r$. If $D\cap V(G^h)$ dominates $C_{i_1}^h,\dots, C_{i_l}^h, C_{j_1}^h, \dots, C_{j_t}^t$ as in definition \ref{rest}, and $D\cap V(G^h)$ has non-zero excess, then we can label one vertex of $D\cap V(G^h)$ by two labels, say $i_1$ and $i_2$, and the rest of the vertices by one distinct label from $\{i_3, \dots, i_l, j_1, \dots, j_t\}$. Thus, in every $G$-fiber with a missing cell, there are at least $r$ vertices of $D$ and at most one such vertex receives two labels.

For any fixed label $i$, $1\leq i \leq k+1$, projecting the vertices of $D$ labeled $i$ onto $H$ produces a dominating set of $H$ which has size at least $\gamma(H)$. Summing over all the labels, we count $(\gamma(G)+1)\gamma(H)$ vertices of $D$. However, those vertices that received two labels are counted twice. Since in every $G$-fiber, if a vertex of $D$ was counted twice, there were at least $r-1$ vertices of $D$ that were counted once. We remove the overcount to conclude,
\begin{align}
\left(\gamma(G)+1\right)\gamma(H)\leq \frac{2}{r}\left|D\right|+\frac{r-1}{r}\left|D\right|=\frac{r+1}{r}\left|D\right|\label{form2}
\end{align}

Putting formulas \eqref{form1} and \eqref{form2} together, we obtain 
\begin{align}
\gamma(G\square H) \geq \min_{1\leq r \leq \gamma(G)}\Big{\{}\max\big{\{}(\gamma(G)-r)\gamma(H),\frac{r}{r+1}\left(\gamma(G)+1\right)\gamma(H)\big{\}}\Big{\}}
\end{align}
which can be minimized to show
\[\gamma(G\square H)\geq \left(\gamma(G)-\sqrt{\gamma(G)}\right)\gamma(H)\]

By Proposition \ref{spanning}, the above inequality holds for any graph in $\A_1$.

\end{proof}

The above argument does not immediately generalize to other classes since graphs in $\D_1$ have the property that for any $h\in H$, every $G$-fiber $G^h$ has either one or no missing cells. This is not true in other classes. For example, if $G\in \D_2$ we could have $G$-fibers with one missing cell, and the overcount argument would not apply.

However, by repeating the undercount argument when $G\in \D_n$ for any non-negative integer $n$, we obtain the same undercount result.

We say a $(r,r+n)$ restraining set $S$ of $G\in \A_n$ is a \emph{minimum restraining set} if $S$ has the minimum restraint over all restraining sets with excess $n$.

\begin{cor}
For any non-negative integer $n$, any graph $G\in \A_n$, and any graph $H$, if $G$ contains a minimum restraining set of size $r$, then $\gamma(G\square H) \geq (\gamma(G)-r)\gamma(H)$.
\end{cor}

Note that this bound is an improvement on the best current bound \cite{ST} for any graph $G$ with minimum restraint at most $\frac{1}{2}\gamma(G)-\frac{1}{2}$.

\section{Acknowledgements}

We would like to thank Bostjan Bre\v{s}ar and Douglas Rall for their patient and insightful comments.

 \bibliographystyle{plain}
 
 \end{document}